\documentclass[12pt,reqno]{amsart}
\usepackage{amsmath,amsfonts,graphicx,amscd,amssymb,epsf,amsthm,enumerate,parskip,mathrsfs,color,ltablex,blindtext}  
 \usepackage[a4paper, margin=1in]{geometry}
\usepackage{times}
 
\theoremstyle{definition}
\newtheorem{theorem}{Theorem}[section]
\newtheorem{proposition}[theorem]{Proposition}

\newtheorem{corollary}[theorem]{Corollary}
\newtheorem{definition}[theorem]{Definition}

\newtheorem{remark}[theorem]{Remark}
 \numberwithin{equation}{section}
\numberwithin{equation}{section}
 
\newcommand{\di}{\displaystyle}

\newcommand{\mb}{\mathbb}
\newcommand{\mf}{\mathbf}
\newcommand{\bg}{\boldsymbol{\gamma}}
\newcommand{\tbg}{\widetilde{\boldsymbol{\gamma}}}
\newcommand{\rvline}{\hspace*{-\arraycolsep}\vline\hspace*{-\arraycolsep}}

\usepackage{hyperref}
\hypersetup{
	colorlinks=true,
	linkcolor=blue,
	citecolor=blue,
}

\begin{document}
  
\title[Generalized Curvatures of Curves in $\mb{R}^n$]{Generalized Curvatures of Curves in $\pmb{\mb{R}^n}$}
 \author{Lee-Peng Teo}
\address{Department of Mathematics, Xiamen University Malaysia\\Jalan Sunsuria, Bandar Sunsuria, 43900, Sepang, Selangor, Malaysia.}
\email{lpteo@xmu.edu.my}
\begin{abstract}
For a curve $\bg:I\to\mb{R}^n$ of order $n-1$, we prove that the generalized curvatures $\kappa_1, \ldots, \kappa_{n-1}$ can be expressed in terms of the leading principal minors of the matrix $\mf{A}(t)^T\mf{A}(t)$, where $\mf{A}(t)$ is the $n\times n$ matrix whose $i$-th column is $\bg^{(i)}(t)$. This gives an efficient algorithm to calculate the curvatures.
 
\end{abstract}
\maketitle

\section{Introduction}
 
  A  parametrized curve  in $\mb{R}^n$  is a smooth function $\bg: I \to\mb{R}^n$ defined on an open interval $I$, such that $\bg'(t)\neq \mf{0}$ for all $t\in I$.  The image of $\bg$ is a one-dimensional Riemannian manifold with metric induced by  the Euclidean metric of $\mb{R}^n$.
 
 If $\bg:I\to\mb{R}^n$ is a  parametrized curve such that the vectors $\bg'(t), \ldots, \bg^{(n-1)}(t)$ are linearly independent, we say that the curve has order $n-1$. For such a curve, one can define the Frenet frame $\{\mf{T}(t), \mf{N}_1(t), \ldots, \mf{N}_{n-1}(t)\}$ for each $t\in I$ using the Gram-Schmidt process and the generalized cross product. The   matrix 
  \[\mf{F}=\begin{bmatrix}\mf{T}(t)& \rvline & \mf{N}_1(t)& \rvline & \cdots & \rvline & \mf{N}_{n-1}(t)\end{bmatrix}\] is an orthogonal matrix with determinant 1. The Frenet-Serret formulas say that with respect to the arclength parameter $s$,
\begin{align*}
\frac{d\mf{T}}{ds}&=\kappa_1 \mf{N}_1,\\
\frac{d\mf{N}_1}{ds}&=-\kappa_1\mf{T}+\kappa_2\mf{N}_2,\\
&\hspace{1cm}\vdots\\
\frac{d\mf{N}_{n-2}}{ds}&=-\kappa_{n-2}\mf{N}_{n-3}+\kappa_{n-1}\mf{N}_{n-1},\\
\frac{d\mf{N}_{n-1}}{ds}&=-\kappa_{n-1}\mf{N}_{n-2}.
\end{align*} The numbers    $\kappa_1, \ldots, \kappa_{n-1}$ are the generalized curvatures of the curve.  

In the classical case where $n=3$, $\kappa_1=\kappa$  and $\kappa_2=\tau$ are respectively  the curvature and the torsion of the curve. The fundamental theorem of the local theory of curves asserts that $\kappa$ and $\tau$ uniquely determine the   curve up to a direct isometry of $\mb{R}^3$. 

Given a space curve $\bg:I\to\mb{R}^3$ in any parametrization, the curvature $\kappa$ and the torsion $\tau$ can be computed by the formulas
\begin{equation}\label{20250707_1}\kappa=\frac{\Vert\bg'\times\bg''\Vert}{\Vert \bg'\Vert^3},\hspace{1cm}\tau=\frac{\left\langle\bg'\times\bg'',\bg'''\right\rangle}{\Vert \bg'\times\bg''\Vert^2}.\end{equation}

In this work, we establish an efficient algorithmn to compute the generalized curvatures $\kappa_1, \ldots,\kappa_{n-1}$ of a parametrized curve $\bg:I\to\mb{R}^n$ of order $n-1$ in terms of  $\bg'(t), \ldots, \bg^{(n)}(t)$,  generalizing the formulas \eqref{20250707_1} to higher dimensions. 
Let $\mf{A}(t)$ be the matrix
\[\mf{A}(t)=\begin{bmatrix}\bg'(t) &\rvline &\bg''(t) & \cdots & \rvline & \bg^{(n)}(t)\end{bmatrix}.\] For $1 \leq i\leq n$, define $\mf{M}_i(t)$ as the $i\times i$ matrix 
\[\mf{M}_i(t)=\begin{bmatrix}\langle\bg'(t),\bg'(t) \rangle& \cdots & \langle\bg'(t), \bg^{(i)}(t)\rangle\\
\vdots & \ddots &\vdots\\ \langle\bg^{(i)}(t),\bg'(t) \rangle& \cdots & \langle\bg^{(i)}(t), \bg^{(i)}(t)\rangle\end{bmatrix},\]which consists of the first $i$ rows and first $i$ columns of the matrix $\mf{B}(t)=\mf{A}(t)^T\mf{A}(t)$. We   show that  for $1\leq i\leq n-1$, 
\begin{equation}\label{20250712_1}\kappa_i(t)^2=\frac{  \det \mf{M}_{i+1}(t) \det \mf{M}_{i-1}(t)}{\left[\Vert\bg'(t)\Vert \det\mf{M}_i(t)\right]^2}.\end{equation}From this, we can obtain $\kappa_i(t)$ using the fact that
   for $1\leq i\leq n-2$, $\kappa_i(t)$ is positive. For  $i=n-1$, $\kappa_{n-1}(t)$ has the same sign as $\det\mf{A}(t)$, and $\left(\det \mf{A}(t)\right)^2=\det\mf{M}_n(t)$.

 \bigskip
\section{Preliminaries}

 First, we give a brief revision of the linear algebra on the Euclidean space $\mb{R}^n$, fixing the notations.
A vector $\mf{v}$ in $\mb{R}^n$ is denoted as  \[\mf{v}=(v_1, \ldots, v_n)\hspace{1cm}\text{or}\hspace{1cm} \mf{v}=\begin{bmatrix} v_1\\\vdots\\v_n\end{bmatrix}.\]In terms of the standard unit vectors $\mf{e}_1, \ldots, \mf{e}_n$,
$\mf{v}=v_1\mf{e}_1+\cdots+v_n\mf{e}_n$.
 If $\mf{u}=(u_1, \ldots, u_n)$ and $\mf{v}=(v_1, \ldots, v_n)$ are two vectors in $\mb{R}^n$, their Euclidean inner product is
   \[\langle \mf{u},\mf{v}\rangle=\sum_{i=1}^nu_iv_i=\mf{u}^T\mf{v}=\mf{v}^T\mf{u}.\]
  The norm of the vector $\mf{v}$ is defined as
  \[\Vert\mf{v}\Vert=\sqrt{\langle\mf{v},\mf{v}\rangle}.\]
  A basis $\{\mf{v}_1, \ldots,\mf{v}_n\}$ of $\mb{R}^n$ is said to be positively oriented if
  \[\det\begin{bmatrix}\mf{v}_1 &\rvline & \cdots &\rvline &\mf{v}_n\end{bmatrix}>0.\]
  
  Given a linearly independent set $\{\mf{v}_1, \ldots, \mf{v}_r\}$ in $\mb{R}^n$, the Gram-Schidmt process produces an orthonormal set $\{\mf{u}_1,\ldots, \mf{u}_r\}$ such that for $1\leq j\leq r$,
  \[\text{span} \{\mf{u}_1, \ldots, \mf{u}_j\}=\text{span} \{\mf{v}_1, \ldots, \mf{v}_j\}.\]
  If
$R_{i,j}=\langle \mf{u}_i,\mf{v}_j$, then $\mf{R}=[R_{i,j}]$ is an upper 
triangular matrix with positive diagonal entries, and
\[\begin{bmatrix}\mf{v}_1 &\rvline & \cdots &\rvline & \mf{v}_r\end{bmatrix}=\begin{bmatrix}\mf{u}_1 &\rvline & \cdots &\rvline & \mf{u}_r\end{bmatrix}\mf{R}.\] 
 
  Classically, the cross product is defined for two vectors in $\mb{R}^3$. It can be generalized in the following way.

 \begin{definition}[Generalized Cross Product]
For $n\geq 2$,   the  {\it generalized cross product} of  $n-1$ vectors in $\mb{R}^n$ is  a $(n-1)$-linear map $\mathscr{P}:(\mb{R}^n)^{n-1}\to\mb{R}^n$.
Given  $n-1$ vectors $\mf{v}_j=(v_{1,j}, v_{2,j}, \ldots, v_{n,j})$, $1\leq j\leq n-1$, their cross product $\mathscr{P}(\mf{v}_1, \ldots,\mf{v}_{n-1})$ is the vector 
\[\mathscr{P}(\mf{v}_1, \ldots,\mf{v}_{n-1})=\det\begin{bmatrix} \mf{v}_{1} &\rvline & \cdots & \rvline &\mf{v}_{n-1} &\rvline& \begin{matrix}\mf{e}_1 \\ \vdots\\\mf{e}_n\end{matrix}\end{bmatrix}=\det\begin{bmatrix} v_{1,1} & \ldots & v_{1,n-1} & \mf{e}_1\\v_{2,1} & \ldots & v_{2,n-1}&\mf{e}_2\\\vdots & \ddots & \vdots & \vdots\\v_{n,1} & \ldots & v_{n, n-1} &\mf{e}_n\end{bmatrix}.\]
\end{definition}In the determinant, $\mf{e}_1, \ldots, \mf{e}_n$ are treated as formal symbols. The components of the cross product $\mathscr{P}(\mf{v}_1, \ldots,\mf{v}_{n-1})$ along $\mf{e}_1, \ldots,\mf{e}_n$ are obtained by computing this determinant using column expansion with respect to the last column. 

By the properties of determinants, the  $(n-1)$-linear map $\mathscr{P}:(\mb{R}^n)^{n-1}\to\mb{R}^n$ is an alternating linear mapping. Namely, if $\sigma$ is a permutation of the set $\{1,2,\ldots, n-1\}$, then 
\[\mathscr{P}\left(\mf{v}_{\sigma(1)}, \mf{v}_{\sigma(2)}, \ldots, \mf{v}_{\sigma(n-1)}\right)=\text{sgn} (\sigma) \mathscr{P}(\mf{v}_1, \ldots,\mf{v}_{n-1}).\]

When $n=3$, $\mathscr{P}(\mf{u}, \mf{v})$ is just the ordinary cross product of $\mf{u}$ and $\mf{v}$.

The generalized cross product has certain properties whose proofs are straightforward generalizations of those for the $n=3$ case.
\begin{proposition}\label{20250707_3}
Let $n\geq 2$ and let $\mf{v}_1$, $\ldots$, $\mf{v}_{n-1}$ be vectors in $\mb{R}^n$. Then the generalized cross product $\mathscr{P}(\mf{v}_1, \ldots, \mf{v}_{n-1})$ is zero if and only if the set $\{\mf{v}_1, \ldots, \mf{v}_{n-1}\}$ is linearly dependent.
\end{proposition}

\begin{proposition} \label{20250710_5}  Let $n\geq 2$, and let $\mf{v}_1, \ldots, \mf{v}_{n-1}, \mf{v}_n$ be vectors in $\mb{R}^n$. Assume that
  $\mf{v}_j=(v_{1,j}, \ldots, v_{n,j})$ for $1\leq j\leq n$. Then
\[\langle \mathscr{P}( \mf{v}_1, \ldots, \mf{v}_{n-1}), \mf{v}_n\rangle =\det \begin{bmatrix} \mf{v}_{1} &\rvline & \cdots & \rvline &\mf{v}_{n-1} &\rvline&\mf{v}_n\end{bmatrix}=\det\begin{bmatrix} v_{i,j}\end{bmatrix}.\]

\end{proposition}
 
Proposition \ref{20250710_5} gives the following important properties of cross product.
\begin{corollary}\label{20250710_6}
Let $n\geq 2$.
If $\mf{v}_1, \ldots, \mf{v}_{n-1}$ are vectors in $\mb{R}^n$, the cross product $\mathscr{P}(\mf{v}_1, \ldots, \mf{v}_{n-1})$ is orthogonal to each of the vectors $\mf{v}_1, \ldots \mf{v}_{n-1}$.
\end{corollary}

\begin{corollary}
Let $n\geq 2$, and let $\mf{v}_1, \ldots, \mf{v}_{n-1}$ be vectors in $\mb{R}^n$. Denote by $\mf{w}=\mathscr{P}(\mf{v}_1, \ldots, \mf{v}_{n-1})$   their cross product. If $\{\mf{v}_1, \ldots, \mf{v}_{n-1}\}$ is a linearly independent set, then $\{\mf{v}_1, \ldots, \mf{v}_{n-1}, \mf{w}\}$ is  a positively oriented basis of $\mb{R}^n$.
\end{corollary}

For the norm of the vector $\mathscr{P}(\mf{v}_1, \ldots, \mf{v}_{n-1})$, we first prove the following general formula.
\begin{proposition}
\label{20250707_4} Let $n\geq 2$. 
If $\{\mf{u}_1, \ldots, \mf{u}_{n-1}\}$ and $\{\mf{v}_1, \ldots, \mf{v}_{n-1}\}$ are two sets of vectors in $\mb{R}^n$, let
\[\mf{u}=\mathscr{P}(\mf{u}_1, \ldots, \mf{u}_{n-1}),\hspace{1cm}\mf{v}=\mathscr{P}(\mf{v}_1, \ldots, \mf{v}_{n-1}).\] Then
\[\langle \mf{u}, \mf{v} \rangle  =\det\begin{bmatrix}\langle \mf{u}_i,\mf{v}_j\rangle\end{bmatrix}=\det\begin{bmatrix}\langle\mf{u}_1,\mf{v}_1\rangle & \cdots & \langle\mf{u}_1, \mf{v}_{n-1}\rangle\\
 \vdots & \ddots &\vdots\\ \langle\mf{u}_{n-1},\mf{v}_1\rangle & \cdots & \langle\mf{u}_{n-1}, \mf{v}_{n-1}\rangle\end{bmatrix}.\]
\end{proposition}
\begin{proof}
If $\mf{v}=\mf{0}$, then $\langle \mf{u}, \mf{v} \rangle  =0$. Proposition \ref{20250707_3}  says  that  $\{\mf{v}_1, \ldots, \mf{v}_{n-1}\}$ is a linearly dependent set. Therefore, there exists a nonzero vector $\mf{c}=(c_1, \ldots, c_{n-1})\in\mb{R}^{n-1}$ such that
\[c_1\mf{v}_1+\cdots+c_{n-1}\mf{v}_{n-1}=\mf{0}.\]
Then
\[\begin{bmatrix} \langle\mf{u}_1,\mf{v}_1\rangle & \cdots & \langle\mf{u}_1, \mf{v}_{n-1}\rangle\\
 \vdots & \ddots &\vdots\\ \langle\mf{u}_{n-1},\mf{v}_1\rangle & \cdots & \langle\mf{u}_{n-1}, \mf{v}_{n-1}\rangle \end{bmatrix}\begin{bmatrix}c_1\\\vdots\\c_{n-1}\end{bmatrix}=\begin{bmatrix}\langle \mf{u}_1, c_1\mf{v}_1+\cdots+c_{n-1}\mf{v}_{n-1}\rangle \\\vdots\\ \langle \mf{u}_{n-1}, c_1\mf{v}_1+\cdots+c_{n-1}\mf{v}_{n-1}\rangle \end{bmatrix}=\mf{0}.\]This implies that the matrix $\di \begin{bmatrix}\langle \mf{u}_i,\mf{v}_j\rangle\end{bmatrix}$ is singular. Hence,
 \[\langle \mf{u}, \mf{v} \rangle  =0=\det\begin{bmatrix}\langle \mf{u}_i,\mf{v}_j\rangle\end{bmatrix}\]holds.
 If $\mf{v}\neq \mf{0}$, then $\langle \mf{v},\mf{v}\rangle >0$. By Proposition \ref{20250710_5},
 \begin{align*}
 \langle \mf{u}, \mf{v}\rangle \langle\mf{v},\mf{v}\rangle &=\det \begin{bmatrix} \mf{u}_1 &\rvline& \cdots &\rvline&\mf{u}_{n-1}  &\rvline& \mf{v} \end{bmatrix}^T\begin{bmatrix} \mf{v}_1 &\rvline& \cdots &\rvline&\mf{v}_{n-1}  &\rvline& \mf{v} \end{bmatrix} \\
 &=\det\begin{bmatrix} \langle\mf{u}_1,\mf{v}_1\rangle & \cdots & \langle\mf{u}_1, \mf{v}_{n-1}\rangle &\langle \mf{u}_1,\mf{v}\rangle\\
 \vdots & \ddots &\vdots&\vdots\\ \langle\mf{u}_{n-1},\mf{v}_1\rangle & \cdots & \langle\mf{u}_{n-1}, \mf{v}_{n-1}\rangle& \langle\mf{u}_{n-1}, \mf{v}\rangle \\ \langle\mf{v},\mf{v}_1\rangle& \cdots & \langle\mf{v}, \mf{v}_{n-1}\rangle& \langle\mf{v}, \mf{v}\rangle \end{bmatrix}.
 \end{align*}
 By Corollary \ref{20250710_6}, $\mf{v}$ is orthogonal to $\mf{v}_1, \ldots,\mf{v}_{n-1}$. Therefore,
 \[\langle\mf{v},\mf{v}_1\rangle= \cdots = \langle\mf{v}, \mf{v}_{n-1}\rangle=0.\]
 Thus,
 \begin{align*}
  \langle \mf{u}, \mf{v}\rangle \langle\mf{v},\mf{v}\rangle &=\det\begin{bmatrix} \langle\mf{u}_1,\mf{v}_1\rangle & \cdots & \langle\mf{u}_1, \mf{v}_{n-1}\rangle &\langle \mf{u}_1,\mf{v}\rangle\\
 \vdots & \ddots &\vdots&\vdots\\ \langle\mf{u}_{n-1},\mf{v}_1\rangle & \cdots & \langle\mf{u}_{n-1}, \mf{v}_{n-1}\rangle& \langle\mf{u}_{n-1}, \mf{v}\rangle \\ 0& \cdots & 0& \langle\mf{v}, \mf{v}\rangle \end{bmatrix}
 \\&=\langle \mf{v},\mf{v}\rangle  \det\begin{bmatrix}\langle\mf{u}_1,\mf{v}_1\rangle & \cdots & \langle\mf{u}_1, \mf{v}_{n-1}\rangle\\
 \vdots & \ddots &\vdots\\ \langle\mf{u}_{n-1},\mf{v}_1\rangle & \cdots & \langle\mf{u}_{n-1}, \mf{v}_{n-1}\rangle\end{bmatrix}.
 \end{align*}Since $\langle \mf{v},\mf{v}\rangle >0$, we find that 
 \[ \langle \mf{u}, \mf{v}\rangle = \det\begin{bmatrix}\langle\mf{u}_1,\mf{v}_1\rangle & \cdots & \langle\mf{u}_1, \mf{v}_{n-1}\rangle\\
 \vdots & \ddots &\vdots\\ \langle\mf{u}_{n-1},\mf{v}_1\rangle & \cdots & \langle\mf{u}_{n-1}, \mf{v}_{n-1}\rangle\end{bmatrix}.\]
\end{proof}

Taking $\mf{u}_j=\mf{v}_j$ for $1\leq j\leq n-1$ in Proposition \ref{20250707_4}, we obtain the norm of the vector $\mathscr{P}(\mf{v}_1, \ldots, \mf{v}_{n-1})$. 
\begin{corollary}
Let $n\geq 2$, and let $\mf{v}_1, \ldots, \mf{v}_{n-1}$ be vectors in $\mb{R}^n$. The  cross product $\mathscr{P}(\mf{v}_1, \ldots, \mf{v}_{n-1})$ is a vector in $\mb{R}^n$  with norm
\[\left\Vert \mathscr{P}(\mf{v}_1, \ldots, \mf{v}_{n-1})\right\Vert=\sqrt{\det\begin{bmatrix}\langle\mf{v}_i,\mf{v}_j\rangle\end{bmatrix}}.\]
\end{corollary}

From this, we obtain the following.
\begin{corollary}\label{20250708_1}
Let $n\geq 2$, and let $\{\mf{u}_1, \ldots, \mf{u}_{n-1}\}$ be an orthonormal set in $\mb{R}^n$. If
\[\mf{u}_n=\mathscr{P}(\mf{u}_1, \ldots, \mf{u}_{n-1}),\]
then $\{\mf{u}_1, \ldots, \mf{u}_{n-1},\mf{u}_n\}$ is a positively oriented orthonormal basis of $\mb{R}^n$.

\end{corollary}
 
Next, we discuss the Euclidean geometry of curves in $\mb{R}^n$.
  Some standard textbooks in this topic are \cite{Gerretsen, Klingenberg, doCarmo, Tapp, Alfred, Wolfgang, Spivak2, Spivak4}.  
  
  A parametrized curve in $\mb{R}^n$ is a smooth function $\bg:I\to\mb{R}^n$ defined on an open interval $I$ with $\bg'(t)\neq \mf{0}$ for all $t\in I$. Fixed a $t_0\in I$, then the arclength function $s:I\to\mb{R}$ defined by
\[s(t)=\int_{t_0}^t \Vert\bg'(\tau)\Vert d\tau \] is a strictly increasing smooth function. If $\widetilde{\bg}:J\to\mb{R}^n$ is the function defined as \[\widetilde{\bg}=\bg\circ s^{-1},\] then $\widetilde{\bg}:J\to\mb{R}^n$ is a reparametrization of $\bg:I\to\mb{R}^n$ by arc-length. 
  We usually use $s$ instead of $t$ as the parameter for a curve $\widetilde{\bg}:J\to\mb{R}^n$ that is parametrized by arclength.

\begin{definition}[The Canonical Matrix]
Let $\bg:I\to\mb{R}^n$ be a  parametrized curve.
The canonical matrix of $\bg$ at $t\in I$ is defined to be the $n\times n$ matrix  
   \[\mf{A}(t)=\begin{bmatrix}\bg'(t) &\rvline & \bg''(t) &\rvline & \cdots & \rvline & \bg^{(n)}(t)\end{bmatrix}.\]
   \end{definition}
If $\tbg:J\to\mb{R}^n$ is a reparametrization of $\bg:I\to\mb{R}^n$, there exists a strictly increasing diffeomorphism $\phi:J\to I$ such that 
\[\tbg(t)=\bg(\phi(t)).\]In particular, $\phi'(t)>0$ for all $t\in J$. 
We have the following.

\begin{proposition}\label{20250710_3}
Let $\bg:I\to\mb{R}^n$ be a parametrized curve. Assume that $\tbg:J\to\mb{R}^n$ is a reparametrization of $\bg:I\to\mb{R}^n$ such that
$\tbg(t)=\bg(\phi(t))$ for a strictly increasing diffeomorphism $\phi:J\to I$. Then for $1\leq j\leq n$, there exist smooth functions $U_{i,j}(t)$, $1\leq i\leq j$ such that
\[\tbg^{(j)}(t) = \sum_{i=1}^{j}U_{i,j}(t)\bg^{(i)}(\phi(t)),\]
with
$U_{j,j}(t)=\phi'(t)^j$. In other words, if  
  \[\mf{A}(t)=\begin{bmatrix}\bg'(t) &\rvline & \bg''(t) &\rvline & \cdots & \rvline & \bg^{(n)}(t)\end{bmatrix}\quad\text{and}\quad \widetilde{\mf{A}}(t)=\begin{bmatrix}\tbg'(t) &\rvline & \tbg''(t) &\rvline & \cdots & \rvline & \tbg^{(n)}(t)\end{bmatrix}\]are the canonical matrices of $\bg(t)$ and $\tbg(t)$, there is an upper triangular matrix $\mf{U}(t)=[U_{i,j}(t)]$ with
  diagonal entries 
  $U_{j,j}(t)=\phi'(t)^j$, $1\leq j\leq n$, such that
  \[\widetilde{\mf{A}}(t)=\mf{A}(\phi(t))\mf{U}(t).\] 
\end{proposition}

The regularity order of a curve is defined in the following way.
 
\begin{definition}[The Regularity Order of a Curve]
Given a curve $\bg:I\to\mb{R}^n$,  if   $k$ is a positive integer such that for all $t\in I$, $\{\bg'(t), \ldots, \bg^{(k)}(t)\}$ is a linearly independent set, we say that the curve $\bg:I\to\mb{R}^n$ is regular of order $k$.   

  \end{definition}
By definition, for any $t\in I$, $\bg'(t)$ is a nonzero vector. Thus, $\{\bg'(t)\}$ is a linearly independent set. Hence, a curve $\bg:I\to\mb{R}^n$ must be regular of order 1. If a curve is regular of order $k$, then $1\leq k\leq n$. 

Now we define the Frenet frame for a curve.
  \begin{definition}[The Frenet Frame]\label{20250713_1}
  Let $n\geq 2$, and let
   $\bg:I\to\mb{R}^n$ be a curve in $\mb{R}^n$ that has order $n-1$.
    For any $t\in I$,  
  the Frenet frame $\{\mf{T}(t), \mf{N}_1(t), \ldots,\mf{N}_{n-1}(t)\}$ is an orthonormal set, where the  vectors $\mf{T}(t), \mf{N}_1(t), \ldots,\mf{N}_{n-2}(t)$ are obtained by applying the Gram-Schmidt process to the linearly independent set $\{\bg'(t), \ldots, \bg^{(n-1)}(t)\}$, and the vector $\mf{N}_{n-1}(t)$ is defined by the cross product
   \[\mf{N}_{n-1}(t)=\mathscr{P}\left(\mf{T}(t), \mf{N}_1(t), \ldots,\mf{N}_{n-2}(t)\right).\]The frame matrix $\mf{F}(t)$  is defined as
   the $n\times n$ matrix    \[  \mf{F}(t)=\begin{bmatrix}\mf{T}(t) &\rvline & \mf{N}_1(t) &\rvline & \cdots & \rvline & \mf{N}_{n-1}(t)\end{bmatrix}.\]
   \end{definition}
   By Corollary \ref{20250708_1}, for any $t\in I$, the Frenet frame $\{\mf{T}(t), \mf{N}_1(t), \ldots,\mf{N}_{n-1}(t)\}$ is a positively oriented orthonormal basis of $\mb{R}^n$. Equivalently, the frame matrix $\mf{F}(t)$ is an orthogonal matrix with determinant 1. The algorithm of the Gram-Schmidt process and the definition of the generalized cross product show  that each of the vectors $\mf{T}(t), \mf{N}_1(t), \ldots,\mf{N}_{n-1}(t)$ is a smooth function of $t$. 
   
   Our definition differs from some literatures which only consider curves of order $n$ in $\mb{R}^n$. They define the Frenet frames as the orthonormal set obtained by applying the Gram-Schmidt process to the set $\{\bg'(t), \bg''(t), \ldots, \bg^{(n)}(t)\}$. The definition we use here produces the same vectors $\mf{T}(t)$, $\ldots$, $\mf{N}_{n-2}(t)$, and produces the vector $\mf{N}_{n-1}(t)$ that might  differ by a sign.
   
One can show that the Frenet frame  is independent of parametrizations.

Given a curve $\bg:I\to\mb{R}^n$  of order $n-1$, let $\mf{A}(t)$ and $\mf{F}(t)$ be respectively the canonical matrix and the frame matrix of $\bg$. Then there exists an $n\times n$  matrix $\mf{R}(t)=[R_{i,j}(t)]$ such that
   \[\mf{A}(t)=\mf{F}(t)\mf{R}(t).\]
 Since the first $n-1$ column vectors of $\mf{F}(t)$ are obtained by applying the Gram-Schmidt process to the first $n-1$ column vectors of $\mf{A}(t)$, $\mf{R}(t)$ is an upper triangular matrix. Moreover, for $1\leq j\leq n-1$, the diagonal entry $ R_{j,j}(t)$ is positive. 
 
Since $\mf{F}(t)$ is an orthogonal matrix and $\mf{R}(t)$ is an upper triangular matrix, $\mf{A}(t)=\mf{F}(t)\mf{R}(t)$ gives a QR-decomposition of the matrix $\mf{A}(t)$. This explains our choice  of the notation for $\mf{R}(t)$.

Now we   define the    generalized curvatures $\kappa_i$, $1\leq i\leq n-1$.
   
\begin{definition}[Generalized Curvatures]\label{20250713_2}
Let $\bg:I\to\mb{R}^n$ be an arclength parametrized curve  of order $n-1$.  For $s\in I$, let $\{\mf{T}(s), \mf{N}_1(s), \ldots, \mf{N}_{n-1}(s)\}$ be the Frenet frame of the curve at $\bg(s)$. The first curvature $\kappa_1(s)$ is defined as
\begin{equation}\label{20250709_3}\kappa_1(s)=\langle\mf{T}'(s),\mf{N}_1(s)\rangle.\end{equation}For $2\leq i\leq n-1$, the $i$-th curvature $\kappa_i(s)$ is defined as
\begin{equation}\label{20250709_4}\kappa_i(s)=\langle \mf{N}_{i-1}'(s),\mf{N}_i(s)\rangle.\end{equation}
\end{definition}
If a curve is not parametrized by arclength, we define the generalized curvatures by its arclength reparametrization.

The Frenet-Serret formulas  is a set of equations that express the derivatives of $\mf{T}(s)$, $\mf{N}_1(s)$, $\ldots$, $\mf{N}_{n-1}(s)$ with respect to $s$ in terms of $\mf{T}(s)$, $\mf{N}_1(s)$, $\ldots$, $\mf{N}_{n-1}(s)$.

\begin{theorem}[Frenet-Serret Formulas] Let $\bg:I\to\mb{R}^n$ be a curve of order $n-1$ that is parametrized by  arclength. For $s\in I$, let $\{\mf{T}(s), \mf{N}_1(s), \ldots, \mf{N}_{n-1}(s)\}$ be the Frenet frame at $\bg(s)$, and let $\kappa_i(s)$, $1\leq i\leq n-1$ be the generalized curvatures. Then
\begin{align*}
\frac{d\mf{T}}{ds}&=\kappa_1\mf{N}_1,\\
\frac{d\mf{N}_1}{ds}&=-\kappa_1\mf{T}+\kappa_2 \mf{N}_2,\\
 &\hspace{1cm}\vdots\\
 \frac{d\mf{N}_j}{ds}&=-\kappa_j\mf{N}_{j-1}+\kappa_{j+1}\mf{N}_{j+1},\hspace{1cm} 2\leq j\leq n-2,\\
 &\hspace{1cm}\vdots\\
 \frac{d\mf{N}_{n-1}}{ds}&=-\kappa_{n-1}\mf{N}_{n-2}.
\end{align*}
In terms of the frame matrix $\di \mf{F}(s)=\begin{bmatrix}\mf{T}(s) &\rvline & \mf{N}_1(s) &\rvline &\cdots &\rvline &\mf{N}_{n-1}(s)\end{bmatrix}$, these Frenet-Serret formulas can be written as
\[\frac{d\mf{F}(s)}{ds}=\mf{F}(s)\mf{C}(s),\]
where $\mf{C}(s)$ is the anti-symmetric matrix
\[\mf{C}(s)=\begin{bmatrix} 0 &-\kappa_1(s) & 0 &\cdots & 0 & 0\\
\kappa_1(s)& 0 & -\kappa_2(s) & \cdots & 0 & 0\\
0 &\kappa_2(s) & 0 &\cdots & 0 & 0\\
\vdots & \vdots & \vdots & \ddots & \vdots & \vdots\\
0 & 0 & 0 &\cdots & 0&-\kappa_{n-1}(s)  \\
0 & 0 & 0 &\cdots & \kappa_{n-1}(s) & 0\end{bmatrix}.\]
\end{theorem}

\bigskip
\section{Generalized Curvatures}
For a curve $\bg:I\to\mb{R}^n$ of order $n-1$, we have defined its generalized curvatures $\kappa_1,\ldots, \kappa_{n-1}$. When $n=3$, $\kappa_1$ is the curvature $\kappa$, $\kappa_2$ is the torsion $\tau$. For general $n$, it has been proved that the curvatures $\kappa_1, \ldots, \kappa_{n-2}$ are positive-valued (see for example,   \cite{Gluck, Gluck_2}). Here we prove this fact in a slightly different way, without the need to define the excess vectors as in \cite{Gluck}. We first prove the following theorem.

\begin{theorem}\label{20250709_9}
Given that $\bg:I\to\mb{R}^n$ is an aclength parametrized curve of order $n-1$, let  $\kappa_1, \ldots, \kappa_{n-1}$ be the generalized curvatures of $\bg:I\to\mb{R}^n$. Denote by $\mf{A}(s)$ and $\mf{F}(s)$ the canonical  matrix and the frame matrix of $\bg$. Let $\mf{R}(s)=[R_{i,j}(s)]$ be the $n\times n$ upper triangular matrix so that
   $\mf{A}(s)=\mf{F}(s)\mf{R}(s)$. Then $R_{1,1}(s)=1$, and for $2\leq j\leq n$, the $j$-th diagonal entry $R_{j,j}(s)$ of $\mf{R}(s)$ is given by
   \[R_{j,j}(s)=\prod_{i=1}^{j-1}\kappa_i(s)=\kappa_1(s)\kappa_2(s)\cdots\kappa_{j-1}(s).\]
   
\end{theorem}
\begin{proof}
By definition, for $1\leq j\leq n$,  we have
\begin{equation}\label{20250709_11}\bg^{(j)}(s)=R_{1,j}(s)\mf{T}(s)+\sum_{i=2}^{j}R_{i,j}(s)\mf{N}_{i-1}(s).\end{equation}In particular,
\[\bg'(s)=R_{1,1}(s)\mf{T}(s).\]
Since $\mf{T}(s)=\bg'(s)$, we find that \[R_{1,1}(s)=1.\] 
For $2\leq j\leq n$, \eqref{20250709_11} gives
\[R_{j,j}(s)=\langle\bg^{(j)}(s),\mf{N}_{j-1}(s)\rangle.\]
Using the fact that $\bg'(s)=\mf{T}(s)$, we have
\[R_{2,2}(s)=\langle \bg''(s),\mf{N}_1(s)\rangle=\langle \mf{T}'(s), \mf{N}_1(s)\rangle=\kappa_1(s).\]
 Now given $2\leq j\leq n-1$,  assume that we have shown that
\begin{equation}\label{20250711_3}R_{j,j}(s)=\kappa_1(s)\ldots\kappa_{j-1}(s).\end{equation}
Differentiating \eqref{20250709_11} with respect to $s$, we have
\begin{align*}
\bg^{(j+1)}(s)=R_{1,j}'(s)\mf{T}(s)+\sum_{i=2}^{j}R_{i,j}'(s)\mf{N}_{i-1}(s)+R_{1,j}(s)\mf{T}'(s)+\sum_{i=2}^{j}R_{i,j}(s)\mf{N}_{i-1}'(s).
\end{align*}
It follows that 
\begin{align*}
R_{j+1,j+1}(s)&=\langle \bg^{(j+1)}(s),\mf{N}_j(s)\rangle\\
&=R_{1,j}(s)\langle \mf{T}'(s),\mf{N}_j(s)\rangle+\sum_{i=2}^{j}R_{i,j}(s)\langle\mf{N}_{i-1}'(s),\mf{N}_j(s)\rangle.\end{align*}
Since $j\geq 2$,  $\langle \mf{T}'(s),\mf{N}_j(s)\rangle=0$. On the other hand,  
$ \langle\mf{N}_{i-1}'(s),\mf{N}_j(s)\rangle=0$ for $2\leq i\leq j-1$. Therefore,
\[R_{j+1,j+1}(s)=R_{j,j}(s)\langle\mf{N}_{j-1}'(s),\mf{N}_j(s)\rangle=R_{j,j}(s)\kappa_{j}(s).\]
By the inductive  hypothesis \eqref{20250711_3}, we conclude that \[R_{j+1,j+1}(s)=\kappa_1(s)\ldots\kappa_{j-1}(s)\kappa_j(s).\]
This completes the proof by induction.
\end{proof}

An immediate consequence of Theorem \ref{20250709_9} is the following result proved in \cite{Gluck}.
\begin{corollary}\label{20250709_14}
Given that $\bg:I\to\mb{R}^n$ is an aclength parametrized curve of order $n-1$, let  $\kappa_1, \ldots, \kappa_{n-1}$ be the generalized curvatures of $\bg:I\to\mb{R}^n$. Denote by $\mf{A}(s)$ and $\mf{F}(s)$ the canonical  matrix and the frame matrix of $\bg$. Let $\mf{R}(s)=[R_{i,j}(s)]$ be the $n\times n$ upper triangular matrix so that
   $\mf{A}(s)=\mf{F}(s)\mf{R}(s)$. Then for $1\leq i\leq n-1$,
   \begin{equation}\label{20251010_1}\kappa_i(s)=\frac{R_{i+1,i+1}(s)}{R_{i,i}(s)}.\end{equation}
\end{corollary}
Since $R_{j,j}(s)>0$ for $1\leq j\leq n-1$, we obtain the following immediately.
 
\begin{theorem}\label{20250709_15}
Let $n\geq 3$, and let $\bg:I\to\mb{R}^n$ be a parametrized curve of order $n-1$. Then the generalized curvatures $\kappa_1, \ldots, \kappa_{n-2}$ are positive-valued.
\end{theorem}
 
\begin{remark}
It can be shown that $R_i(s)$ is the norm of the excess vector $E_i(s)$ defined in \cite{Gluck}. Hence, the formula \eqref{20251010_1} is essentially Theorem {3.1} in \cite{Gluck}.
\end{remark}

For the sign of the generalized curvature $\kappa_{n-1}$, we have the following.
\begin{theorem}\label{20250711_9}
Given that $\bg:I\to\mb{R}^n$ is a parametrized curve of order $n-1$, let
 $\mf{A}(t)$ be the canonical matrix of $\bg$, and
let $\kappa_{n-1}(t)$ be the $(n-1)$-th generalized curvature of  the curve at $\bg(t)$.  Then $\kappa_{n-1}(t)$ has the same sign as 
$\det \mf{A}(t)$.  In particular, $\kappa_{n-1}(t)=0$ if and only if $\{\bg'(t), \ldots, \bg^{(n)}(t)\}$ is a linearly dependent set.
\end{theorem}
\begin{proof}
It is sufficient to prove this theorem under the   assumption that $\bg:I\to\mb{R}^n$ is an arclength parametrization. Let $\mf{F}(s)$ be the frame matrix of $\bg$, and
 let $\mf{R}(s)$ be the $n\times n$ upper triangular matrix such that
 \[\mf{A}(s)=\mf{F}(s)\mf{R}(s).\]
 Since $\det \mf{F}(s)=1$, \[\det \mf{A}(s)=\det \mf{R}(s)=\prod_{i=1}^nR_{i,i}(s).\] Since $R_{i,i}(s)>0$ for $1\leq i\leq n-1$, 
 $\det \mf{A}(s)$ has the same sign as $R_{n,n}(s)$. By Theorem \ref{20250709_9},
 \[R_{n,n}(s)=\kappa_1(s)\ldots \kappa_{n-2}(s)\kappa_{n-1}(s).\] By Theorem \ref{20250709_15}, $\kappa_i(s)>0$ if $1\leq i\leq n-2$. Hence, $R_{n,n}(s)$ has the same sign as $\kappa_{n-1}(s)$. This implies that $\kappa_{n-1}(s)$ has the same sign as $\det \mf{A}(s)$.   In particular, $\kappa_{n-1}(s)=0$ if and only if $\det\mf{A}(s)=0$, if and only if $\{\bg'(s), \ldots, \bg^{(n)}(s)\}$ is a linearly dependent set.
\end{proof}

To find the generalized curvatures of     a curve $\bg:I\to\mb{R}^n$ of order $n-1$ that is parametrized by arclength, we can apply the Gram-Schmidt process to the set $\{\bg'(s), \ldots, \bg^{(n-1)}(s)\}$. In the process, we can compute   the coefficients $R_{i,i}(s)=\langle \bg^{(i)}(s),\mf{N}_{i-1}(s)\rangle$ for $2\leq i\leq n-1$.  The coefficient $R_{n,n}(s)=\langle \bg^{(n)}(s),\mf{N}_{n-1}(s)\rangle$  can then be computed using the fact that
 \[\prod_{i=1}^nR_{i,i}=\det \mf{A}(s).\] This allows us to  compute the generalized curvatures $\kappa_1(s), \ldots, \kappa_{n-1}(s)$ using Corollary \ref{20250709_14}.
 
 Although in principle one can always reparametrize a curve $\bg:I\to\mb{R}^n$ by arclength, this is by no means a simple task. 
It is desirable to have  formulas for  the generalized curvatures of a parametrized curve $\bg:I\to\mb{R}^n$ of order $n-1$ purely in terms of $\bg'(t), \ldots, \bg^{(n)}(t)$, generalizing the $n=3$ formulas 
\begin{equation}\label{202507010_2}\kappa_1(t)=\frac{\Vert\bg'(t)\times\bg''(t)\Vert}{\Vert \bg'(t)\Vert^3},\hspace{1cm}\kappa_2(t)=\frac{\left\langle\bg'(t)\times\bg''(t),\bg'''(t)\right\rangle}{\Vert \bg'(t)\times\bg''(t)\Vert^2}. \end{equation} 
Note that 
\begin{align*}\Vert\bg'(t)\times \bg''(t)\Vert^2&=\det\begin{bmatrix} \langle \bg'(t),\bg'(t)\rangle & \langle \bg'(t),\bg''(t)\rangle\\\langle \bg''(t),\bg'(t)\rangle & \langle \bg''(t),\bg''(t)\rangle\end{bmatrix},\\
\langle\bg'(t)\times\bg''(t), \bg'''(t)\rangle &=\det\begin{bmatrix}\bg'(t)&\rvline&\bg''(t)&\rvline&\bg'''(t)\end{bmatrix}.\end{align*}
To shed further lights, note that
\begin{align*}
\left(\det\begin{bmatrix}\bg'(t)&\rvline&\bg''(t)&\rvline&\bg'''(t)\end{bmatrix}\right)^2&=\det\begin{bmatrix}\hspace{0.2cm}\begin{matrix}\bg'(t)^T\\\hline \bg''(t)^T \\\hline \bg'''(t)^T\end{matrix}\hspace{0.2cm}\end{bmatrix}\begin{bmatrix}\bg'(t)&\rvline&\bg''(t)&\rvline&\bg'''(t)\end{bmatrix}\\&=\det \begin{bmatrix} \langle \bg'(t),\bg'(t)\rangle & \langle \bg'(t),\bg''(t)\rangle& \langle \bg'(t),\bg'''(t)\rangle\\\langle \bg''(t),\bg'(t)\rangle & \langle \bg''(t),\bg''(t)\rangle& \langle \bg''(t),\bg'''(t)\rangle\\\langle \bg'''(t),\bg'(t)\rangle & \langle \bg'''(t),\bg''(t)\rangle& \langle \bg'''(t),\bg'''(t)\rangle\end{bmatrix}.\end{align*}
 Motivated by this, for any parametrized curve   $\bg:I\to\mb{R}^n$,
 let 
$\mf{A}(t)$ be its canonical matrix, and consider the matrix 
\[\mf{B}(t)=\mf{A}(t)^T\mf{A}(t).\]The $(i,j)$-component of $\mf{B}(t)$ is
\[B_{i,j}(t)=  \langle\bg^{(i)}(t), \bg^{(j)}(t)\rangle.\]For $1\leq i\leq n$, let $\mf{M}_i(t)$ be the $i\times i$ matrix 
\[\mf{M}_i(t)=\begin{bmatrix}\langle\bg'(t),\bg'(t) \rangle& \cdots & \langle\bg'(t), \bg^{(i)}(t)\rangle\\
\vdots & \ddots &\vdots\\ \langle\bg^{(i)}(t),\bg'(t) \rangle& \cdots & \langle\bg^{(i)}(t), \bg^{(i)}(t)\rangle\end{bmatrix},\]which consists of the first $i$ rows and first $i$ columns of $\mf{B}(t)$. The determinant of $\mf{M}_i(t)$ is called the $i$-th leading principal minor of $\mf{B}(t)$.

Our main result is that the generalized curvatures $\kappa_1(t)$, $\ldots$, $\kappa_{n-1}(t)$ can be expressed in terms of the determinants of $\mf{M}_i(t)$. If $\bg:I\to\mb{R}^{n-1}$ is a curve of order $n-1$, for $1\leq i\leq n-1$, since $\{\bg'(t), \ldots,\bg^{i}(t)\}$ is a linearly independent set.
This implies that the matrix
\[\mf{M}_i(t)=\begin{bmatrix} \bg'(t) &\rvline & \cdots &\rvline &\bg^{(i)}(t)\end{bmatrix}^T\begin{bmatrix} \bg'(t) &\rvline & \cdots &\rvline &\bg^{(i)}(t)\end{bmatrix}\]
 is   positive definite, and so $\det\mf{M}_i(t)>0$. In fact, $\det\mf{M}_i(t)$ is the square of the volume of the parallelepiped spanned by $\bg'(t), \ldots, \bg^{(i)}(t)$. 
 
\begin{theorem}\label{20250709_16}
Let $n\geq 2$. Given  a parametrized curve $\bg:I\to\mb{R}^n$ of order $n-1$, let $\mf{A}(t)$ be its    canonical matrix, and let $\kappa_1(t), \ldots,\kappa_{n-1}(t)$ be the generalized curvatures.  For $1\leq i\leq n$, denote by $\mf{M}_i(t)$ the matrix  
\[\mf{M}_i(t)=\begin{bmatrix}\langle\bg'(t),\bg'(t) \rangle& \cdots & \langle\bg'(t), \bg^{(i)}(t)\rangle\\
\vdots & \ddots &\vdots\\ \langle\bg^{(i)}(t),\bg'(t) \rangle& \cdots & \langle\bg^{(i)}(t), \bg^{(i)}(t)\rangle\end{bmatrix}.\]
Then for $n\geq 3$,
\[\kappa_1(t)=\frac{\sqrt{\det\mf{M}_2(t)}}{\Vert\gamma'(t)\Vert^3};\]for $2\leq i\leq n-2$, 
\[\kappa_i(t)=\frac{\sqrt{ \det \mf{M}_{i+1}(t) \det \mf{M}_{i-1}(t)}}{\Vert\bg'(t)\Vert \det\mf{M}_i(t)};\] and
for $n\geq 2$,
\[\kappa_{n-1}(t)=\frac{\det\mf{A}(t)}{ \Vert\bg'(t)\Vert\det\mf{M}_{n-1}(t) }\sqrt{  \det \mf{M}_{n-2}(t)}.\]

\end{theorem}

\begin{proof}

Let $s:I\to J$ be an arclength function of $\bg:I\to\mb{R}^n$. Then 
\[s'(t)=\Vert\bg'(t)\Vert.\] Let $\tbg:J\to\mb{R}^n$ be the arclength reparametrization of $\bg: I\to\mb{R}^n$ so that
\[\bg(t)=\tbg(s(t)).\]
Denote by $\widetilde{\mf{A}}(s)$   the canonical matrix of $\tbg(s)$. By Proposition \ref{20250710_3}, there exists an upper triangular matrix $\mf{U}(t)=[U_{i,j}(t)]$ with diagonal entries $U_{j,j}(t)=s'(t)^j$, $1\leq j\leq n$, such that
\[\mf{A}(t)=\widetilde{\mf{A}}(s(t))\mf{U}(t).\]
Let $\{\mf{T}(t), \mf{N}_1(s), \ldots, \mf{N}_{n-1}(t)\}$ be the Frenet frame at $\bg(t)=\tbg(s(t))$, and let $\mf{F}(t)=\widetilde{\mf{F}}(s(t))$ be the corresponding frame matrix. Then there are   upper triangular matrices $\mf{R}(t)=[R_{i,j}(t)]$ and  $\widetilde{\mf{R}}(s)=[\widetilde{R}_{i,j}(s)]$ such that
\[\mf{A}(t)=\mf{F}(t)\mf{R}(t)\hspace{1cm}\text{and}\hspace{1cm}\widetilde{\mf{A}}(s)=\widetilde{\mf{F}}(s)\widetilde{\mf{R}}(s).\]
It follows  that
\[\mf{R}(t)=\widetilde{\mf{R}}(s(t))\mf{U}(t).\]
Since the matrices $\widetilde{\mf{R}}(s)$ and $\mf{U}(t)$ are upper triangular, we find that for $1\leq j\leq n$,
\[R_{j,j}(t)=\widetilde{R}_{j,j}(s(t))U_{j,j}(t)=\widetilde{R}_{j,j}(s(t))s'(t)^j.\]
When $j=1$, Theorem \ref{20250709_9} says that $\widetilde{R}_{1,1}(s)=1$. Therefore, we have
\[R_{1,1}(t)=s'(t).\]
For $1\leq i\leq n-1$, Corollary \ref{20250709_14} says that
\begin{equation}\label{20250711_7}\kappa_i(t)=\frac{\widetilde{R}_{i+1,i+1}(s(t))}{\widetilde{R}_{i,i}(s(t))}=\frac{1}{s'(t)}\frac{R_{i+1,i+1}(t)}{R_{i,i}(t)}.\end{equation}
Now consider the matrix $\mf{B}(t)=\mf{A}(t)^T\mf{A}(t)$.
Since $\mf{A}(t)=\mf{F}(t)\mf{R}(t)$ and $\mf{F}(t)$ is an orthogonal matrix, we have
\[\mf{B}(t)=\mf{R}(t)^T\mf{F}(t)^T\mf{F}(t)\mf{R}(t)=\mf{R}(t)^T \mf{R}(t).\]For  fixed $1\leq i\leq n$, we partition the matrix $\mf{R}(t)$ into 4 blocks
\[\mf{R}(t)=\begin{bmatrix} \hspace{0.2cm}\begin{matrix}\mf{V}_{i,1}(t) &\rvline & \mf{V}_{i,2}(t)\\\hline\mf{V}_{i,3}(t) &\rvline & \mf{V}_{i,4}(t) \end{matrix}\hspace{0.2cm}\end{bmatrix},\]such that
$\mf{V}_{i,1}$ is an $i\times i$ matrix. Since $\mf{R}(t)$ is upper triangular, $\mf{V}_{i,1}(t)$ is an upper triangular matrix and $\mf{V}_{i,3}(t)=\mf{0}$ is the zero matrix. Therefore,
 \[\mf{B}(t)=\begin{bmatrix} \hspace{0.2cm}\begin{matrix} \mf{V}_{i,1}(t)^T &\rvline & \mf{0}\\\hline\mf{V}_{i,2}(t)^T &\rvline & \mf{V}_{i,4}(t)^T\end{matrix}\hspace{0.2cm}\end{bmatrix}\begin{bmatrix}  \hspace{0.2cm}\begin{matrix}\mf{V}_{i,1}(t) &\rvline & \mf{V}_{i,2}(t)\\\hline\mf{0} &\rvline & \mf{V}_{i,4}(t)\end{matrix}\hspace{0.2cm}\end{bmatrix},\]
and we obtain
 \[\mf{M}_i(t)= \mf{V}_{i,1}(t)^T \mf{V}_{i,1}(t).\]
 It follows that
 \[\det \mf{M}_i(t)=\left(\det  \mf{V}_{i,1}(t)\right)^2.\]
 By definition, $\mf{V}_{i,1}(t)$ is an upper triangular $i\times i$ matrix with diagonal entries $R_{1,1}(t), \ldots, R_{i,i}(t)$. Hence,
 \[\det \mf{M}_i(t)=\left(\prod_{j=1}^iR_{j,j}(t)\right)^2.\]This gives
 \[\det\mf{M}_1(t)=R_{1,1}(t)^2=s'(t)^2,\] and
 for $1\leq i\leq n-1$,
 \begin{equation}\label{20250711_8}R_{i+1,i+1}(t)^2=\frac{\det \mf{M}_{i+1}(t)}{\det\mf{M}_i(t)}.\end{equation}
 Using the fact that $ s'(t) =\Vert\bg'(t)\Vert$, \eqref{20250711_7} and \eqref{20250711_8} give
 \[\kappa_1(t)^2=\frac{1}{s'(t)^2}\frac{R_{2,2}(t)^2}{R_{1,1}(t)^2}=\frac{1}{\Vert\bg'(t)\Vert^4}\frac{\det\mf{M}_2(t)}{\det\mf{M}_1(t)}=\frac{\det\mf{M}_2(t)}{\Vert \bg'(t)\Vert^6}.\]
 Theorem \ref{20250709_15} says that if $n\geq 3$, $\kappa_1(t)>0$. Hence, if $n\geq 3$,
 \[\kappa_1(t)=\frac{\sqrt{\det\mf{M}_2(t)}}{\Vert\bg'(t)\Vert^3}.\]
 If $2\leq i\leq n-1$, \eqref{20250711_7}  and \eqref{20250711_8} give
 \begin{equation}\label{20250711_10}\kappa_i(t)^2=\frac{1}{\Vert\bg'(t)\Vert^2}\frac{ \det \mf{M}_{i+1}(t) \det \mf{M}_{i-1}(t)}{[\det\mf{M}_i(t)]^2}.\end{equation}
If $2\leq i\leq n-2$, Theorem \ref{20250709_15} says that $\kappa_i(t)>0$. Therefore, when $2\leq i\leq n-2$,
\[\kappa_i(t)=\frac{\sqrt{ \det \mf{M}_{i+1}(t) \det \mf{M}_{i-1}(t)}}{\Vert\bg'(t)\Vert \det\mf{M}_i(t)}.\] Finally, we notice that
\[\det\mf{M}_n(t)=\det\mf{B}(t)=\left(\det \mf{A}(t)\right)^2.\]
Therefore, \eqref{20250711_10} gives
\[\kappa_{n-1}(t)^2=\frac{1}{\Vert\bg'(t)\Vert^2}\frac{[\det\mf{A}(t)]^2\det \mf{M}_{n-2}(t)}{[\det\mf{M}_{n-1}(t)]^2}.\]
By Theorem \ref{20250711_9}, $\kappa_{n-1}(t)$ has the same sign as $\det\mf{A}(t)$. It follows that
\[\kappa_{n-1}(t)=\frac{\det\mf{A}(t)}{ \Vert\bg'(t)\Vert\det\mf{M}_{n-1}(t) }\sqrt{  \det \mf{M}_{n-2}(t)}.\]
\end{proof}

Note that since $\det\mf{M}_1(t)=\Vert\bg'(t)\Vert^2$, if $n\geq 4$, we can simplify $\kappa_2(t)$ to
\[\kappa_2(t)=\frac{\sqrt{\det\mf{M}_3(t)}}{\det\mf{M}_2(t)}.\]
If $n=3$, we have the classical formula
\[\kappa_2(t)=\frac{\det\mf{A}(t)}{\det\mf{M}_2(t)}.\]

A disguised form of the results in Theorem \ref{20250709_16} have been obtained in \cite{Gerretsen, Gluck, Gutkin}, where $\sqrt{\det\mf{M}_i(t)}$ is written as the volume   of the parallelepiped spanned by $\bg'(t), \ldots, \bg^{(i)}(t)$. In fact, if $\{\mf{v}_1, \ldots, \mf{v}_r\}$ is a linearly independent set in $\mb{R}^n$,   $\{\mf{u}_1, \ldots, \mf{u}_r\}$ is the orthonormal set obtained by applying the Gram-Schmidt process to $\{\mf{v}_1, \ldots, \mf{v}_r\}$, and
\[R_{i,j}=\langle \mf{u}_i,\mf{v}_j\rangle,\] then $[R_{i,j}]$ is an upper triangular matrix with positive diagonal entries $R_{j,j}$, $1\leq j\leq r$. Moreover, for $1\leq j\leq r$,
\[\mf{v}_j=\sum_{i=1}^{j}R_{i,j}\mf{u}_i=\sum_{i=1}^{j-1}R_{i,j}\mf{u}_i+R_{j,j}\mf{u}_j.\]
This implies that
\[\mf{w}_j=\mf{v}_j-\sum_{i=1}^{j-1}\langle \mf{u}_i,\mf{v}_j\rangle\mf{u}_i=R_{j,j}\mf{u}_j.\]
Since $\mf{w}_j$ is the component of $\mf{v}_j$ perpendicular to the subspace spanned by $\mf{v}_1, \ldots, \mf{v}_{j-1}$, we find that the volume of the parallelepiped spanned by $\mf{v}_1, \ldots, \mf{v}_r$ is
\[\Vert\mf{w}_1\Vert \cdots\Vert\mf{w}_r\Vert=R_{1,1}\cdots R_{r,r}.\]
 
 Theorem \ref{20250709_16} provides an efficient way to compute the generalized curvatures of a curve $\bg:I\to\mb{R}^n$ under any parametrization. One first computes the canonical matrix $\mf{A}(t)$, and then the matrix $\mf{B}(t)=\mf{A}(t)^T\mf{A}(t)$.  From this, one can extract the  matrices $\mf{M}_i(t)$ and compute the curvatures $\kappa_1(t), \ldots, \kappa_{n-1}(t)$ using their determinants and the formulas given in  Theorem \ref{20250709_16}.  In the process, there is no need to apply the Gram-Schmidt algorithm.
 
 As an example, let us consider the curve in $\mb{R}^4$ given by $\bg:\mb{R}\to\mb{R}^4$,
 \[\bg(t)=(t,t^2, t^3, t^4).\]
 Using any computer algebra that can perform symbolic computation, we can easily  find that the canonical matrix of $\bg:I\to\mb{R}^4$ is
 \[\mf{A}(t)=\begin{bmatrix} 1 & 0 & 0 & 0\\2t & 2 & 0 & 0\\
 3t^2 & 6t & 6 & 0\\
 4t^3 & 12t^2 & 24t & 24\end{bmatrix}.\] It follows that
 \[\mf{B}(t)=\mf{A}(t)^T\mf{A}(t)= \begin{bmatrix} 1+4 t^2+9 t^4+16 t^6 & 4 t+18 t^3+48 t^5 & 18 t^2+96 t^4 & 96 t^3\\ 4 t+18 t^3+48 t^5 & 4+36 t^2+144 t^4 & 36 t+288 t^3 & 288 t^2\\ 18 t^2+96 t^4 & 36 t+288 t^3 & 36+576 t^2 & 576 t\\ 96 t^3 & 288 t^2 & 576 t & 576 \end{bmatrix}.
\]
From this, we obtain the first 3 leading principal minors of $\mf{B}(t)$, and the determinant of $\mf{A}(t)$ as  
\begin{align*}\det\mf{M}_1(t)&=1+4 t^2+9 t^4+16 t^6=\Vert\bg'(t)\Vert^2,\\
\det\mf{M}_2(t)&=4+36 t^2+180 t^4+256 t^6+144 t^8,\\
\det\mf{M}_3(t)&=144+2304 t^2+5184 t^4+2304 t^6,\\
\det\mf{A}(t)&=288.
\end{align*}By Theorem \ref{20250709_16}, the generalized curvatures $\kappa_1(t), \kappa_2(t), \kappa_3(t) $ are given by
\begin{align*}
\kappa_1(t)&= \frac{\sqrt{\det\mf{M}_2(t)}}{\Vert \bg'(t)\Vert^3}=\frac{\sqrt{4+36 t^2+180 t^4+256 t^6+144 t^8}}{\left(1+4 t^2+9 t^4+16 t^6\right)^{\frac{3}{2}}},\\
\kappa_2(t)&=\frac{ \sqrt{\det\mf{M}_3(t)}}{  \det\mf{M}_2(t)}=\frac{\sqrt{144+2304 t^2+5184 t^4+2304 t^6}}{4+36 t^2+180 t^4+256 t^6+144 t^8},\\
\kappa_3(t)&=\frac{\det\mf{A}(t)\sqrt{\det\mf{M}_2(t)}}{\Vert\bg'(t)\Vert \det\mf{M}_3(t)}\\&=\frac{288}{144+2304 t^2+5184 t^4+2304 t^6}\sqrt{\frac{4+36 t^2+180 t^4+256 t^6+144 t^8}{1+4 t^2+9 t^4+16 t^6}}.
\end{align*}

\bigskip
\section{Generalizations}
In this section, we consider the general case of a curve $\bg:I\to\mb{R}^n$ which does not necessary have order $n-1$. 

Given a curve $\bg:I\to\mb{R}^n$, we define the canonical matrix $\mf{A}(t)$ and consider the matrix 
$\mf{B}(t)=\mf{A}(t)^T\mf{A}(t)$. One can then compute  $\det\mf{A}(t)$ and the leading principal minors of $\mf{B}(t)$ given by
\[\det\mf{M}_i(t)=\det \begin{bmatrix}\langle\bg'(t),\bg'(t) \rangle& \cdots & \langle\bg'(t), \bg^{(i)}(t)\rangle\\
\vdots & \ddots &\vdots\\ \langle\bg^{(i)}(t),\bg'(t) \rangle& \cdots & \langle\bg^{(i)}(t), \bg^{(i)}(t)\rangle\end{bmatrix},\hspace{1cm}1\leq i\leq n.\]
For $1\leq i\leq n$, define the set $I_i$ as 
\[I_i=\left\{t\in I |  \det\mf{M}_i(t)=0\right\}.\]
Since $\det\mf{M}_i(t)$ is a continuous function, $I_i$ must be a closed subset of $I$. Note that $t\in I_i$ if and only if the set $\{\bg'(t), \ldots, \bg^{(i)}(t)\}$ is linearly dependent. Hence, $I_1=\emptyset$, and 
\[I_1\subset I_2\subset \cdots \subset I_n\subset I.\] Since $\{\bg'(t), \ldots, \bg^{(n+1)}(t)\}$ must be a linearly dependent set, we extrapolate and define $I_{n+1}=I$. 

For any $1\leq k\leq n$, 
the curve has order $k$ if and only if $I_k=\emptyset$. 

Since $I_1=\emptyset$ and $I_{n+1}=I$,  there exists $1\leq r\leq n$ such that $I_{r+1}=I$ and $I_r\neq I$, then $\widetilde{I}_r=I\setminus I_r$ is an open nonempty  subset of real numbers. It can be written as a disjoint union of countably many open intervals. We can restrict the curve $\bg$ to each of these open intervals  and consider them separately.
 Thus, it is sufficient to consider a curve $\bg:I\to\mb{R}^n$   so that for all $t\in I$,    $\bg'(t), \ldots, \bg^{(r)}(t)$ are linearly independent, but  $\bg'(t), \ldots, \bg^{(r)}(t), \bg^{(r+1)}(t)$ are linearly dependent. 
 
 If $r=n-1$ or $r=n$, we can define the Frenet frame as before. 
 
 If $1\leq r\leq n-2$, we can still define the orthonormal set $\{\mf{T}(t), \mf{N}_1(t), \ldots, \mf{N}_{r-1}(t)\}$ by applying the Gram-Schmidt process to the linearly independent set $\bg'(t), \ldots, \bg^{(r)}(t)$. Then one can define the curvatures $\kappa_1(t), \ldots, \kappa_{r-2}(t)$ as before. They are positive valued.
 
  Since $\bg'(t), \ldots, \bg^{(r)}(t)$ are linearly independent, but  $\bg'(t), \ldots, \bg^{(r)}(t), \bg^{(r+1)}(t)$ are linearly dependent, we find that
 \[\bg^{(r+1)}(t)\in\text{span} \{\bg'(t),\ldots,\bg^{(r)}(t)\}=\text{span} \{\mf{T}(t), \mf{N}_1(t), \ldots, \mf{N}_{r-1}(t)\}.\]
 This implies that 
 \[ \mf{N}_{r-1}(t)\in \text{span} \{\bg'(t),\ldots,\bg^{(r)}(t)\},\]
 and so
 \[\mf{N}_{r-1}'(t)\in \text{span} \{\bg'(t),\ldots,\bg^{(r)}(t),\bg^{(r+1)}(t)\}=\text{span} \{\mf{T}(t), \mf{N}_1(t), \ldots, \mf{N}_{r-1}(t)\}.\]
 The Frenet-Serret formulas are
 \begin{align*}
 \frac{d\mf{T}}{ds}&=\kappa_1\mf{N}_1,\\
 \frac{d\mf{N}_1}{ds}&=-\kappa_1\mf{T}+\kappa_2\mf{N}_2,\\
 \frac{d\mf{N}_i}{ds}&=-\kappa_{i}\mf{N}_{i-1}+\kappa_{i+1}\mf{N}_{i+1},\hspace{1cm}2\leq i\leq r-2,\\
 \frac{d\mf{N}_{r-1}}{ds}&=-\kappa_{r-1}\mf{N}_{r-2}.
 \end{align*}
 The formulas for $\kappa_i(t)$, $1\leq i\leq r-1$, given in Theorem \ref{20250709_16}, still hold.

\end{document}